\documentclass[11pt]{article}
\usepackage{amsmath,amssymb,amsthm}%-------------- 数学宏包
\usepackage[margin=2.2cm]{geometry}%--------------- 页边距
\usepackage{titlesec,hyperref}%-------------- 标题样式，标号超链接
\usepackage{fancyhdr}  
\newtheorem{theo}{Theorem}[section]
\newtheorem{lemm}[theo]{Lemma}
\newtheorem{defi}[theo]{Definition}

\newtheorem{prop}[theo]{Proposition}
\newtheorem{rema}[theo]{Remark}
\numberwithin{equation}{section}

\begin{document}

\title{Global well-posedness for the full compressible Navier-\\ Stokes equations }

\author{ \footnote{Email Address:lijl29@mail2.sysu.edu.cn (J. Li),\  mcsyzy@mail.sysu.edu.cn (Z. Yin),\  pingxiaozhai@163.com (X. Zhai).}
Jinlu $\mbox{Li}^1$   \quad Xiaoping $\mbox{Zhai}^1$  \quad and \quad Zhaoyang $\mbox{Yin}^{1,2}$ \\
 \small $^1\mbox{Department}$ of Mathematics, Sun Yat-sen University, \\
 \small Guangzhou, 510275, China\\
\small $^2\mbox{Faculty}$ of Information Technology,
\\ \small Macau University of Science and Technology, Macau, China}

\date{}
\maketitle

\begin{abstract}
In this paper, we mainly study the Cauchy problem for the full compressible Navier-Stokes equations in Sobolev spaces. We establish the global well-posedness of the equations with small initial data by using Friedrich's method and compactness arguments.

\vspace*{1em}


\vspace*{1em}

\end{abstract}

\section{Introduction and the main result}
\quad \quad The full compressible Navier-Stokes equations can be written in the sense of Eulerian coordinates in $\mathbb{R}^d$ as follows
\begin{eqnarray}\label{1.1}
\left\{\begin{aligned}
&\partial_t\rho+\mathrm{div}(\rho u)=0,\quad\quad(t,x)\in\mathbb{R}^+\times\mathbb{R}^d, \\
&\partial_t(\rho u)+\mathrm{div}(\rho u\otimes u)+\nabla P=\mathrm{div} \mathbb{S},\quad\quad(t,x)\in\mathbb{R}^+\times\mathbb{R}^d,\\
&\partial_t(\rho Q(\theta))+\mathrm{div}(\rho uQ(\theta))+\mathrm{div} q=\mathbb{S}:\nabla u-\theta P_{\theta}\mathrm{div} u, \quad\quad(t,x)\in\mathbb{R}^+\times\mathbb{R}^d,\\
\end{aligned}\right.
\end{eqnarray}
where $\rho(t,x)$, $u(t,x)$ and $\theta(t,x)$ stand for the density, the velocity and the temperature of the fluid. Here $P(\rho,\theta)=P_e(\rho)+\theta P_{\theta}(\rho)$, where $\theta P_{\theta}(\rho)$, $P_{e}(\rho)$ denote thermal pressure and elastic pressure respectively. $\mathbb{S}=\mu(\rho)(\nabla u+(\nabla u)^{T})$ represents the viscous stress tensor, which characterizes the measure of resistance of the fluid to flow. $Q(\theta)=\int^{\theta}_0c_v(z)\mathrm{d} z$ and $c_v(\theta)$ represents the specific heat at constant volume. The heat conduction $q$ is given by $q=-\kappa(\theta)\nabla \theta$ (see e.g. the introduction of \cite{Feireisl} ).

 This model has been studied by many mathematicians and made a big progress in the past several years, due to the significance of the physical background.
There have been lots of works about the existence, uniqueness, regularity and asymptotic behavior of the solutions. While, because of the stronger nonlinearity in \eqref{1.1} compared with the Navier-Stokes equations for isentropic flow (no temperature equation), many known mathematical results focused only on the absence of vacuum (vacuum means $\rho=0$). The local existence and uniqueness of smooth solutions for the system \eqref{1.1} were proved by Nash \cite{nash} for smooth initial data without vacuum. Itaya in \cite{itaya} considered the Cauchy problem of the Navier-Stokes equations in $\mathbb{R}^3$ with heat-conducting fluid and obtained the local classical solutions in H$\ddot{\mathrm{o}}$lder spaces. The same result was obtained by Tani in \cite{tani} for IBVP with $\inf \rho_0>0$. Later on, Matsumura and Nishida \cite{mn} proved the global well-posedness for smooth data close to equilibrium, see also \cite{ks} for one dimension. On the existence, asymptotic behavior of the weak solutions of the full compressible Navier-Stokes equations with $\inf \rho_0>0$, please refer for instance to \cite{jiangsong1998}, \cite{jiangsong1999} for the existence of weak solutions in 1D and for the existence of spherically symmetric weak solutions in $\mathbb{R}^d(d=2,3)$, and refer to \cite{hoff} for the existence of spherically and cylindrically symmetric weak solutions in $\mathbb{R}^3$, and refer to \cite{Feireisl3} for the existence of variational solutions in a bounded domain in $\mathbb{R}^d(d=2,3).$

In the presence of vacuum, Feireisl in \cite{Feireisl} got the existence of so-called variational solutions in $\mathbb{R}^d\ (d\ge 2)$. The temperature equation in \cite{Feireisl} is satisfied only as an inequality in the sense of distributions. In order that the equations are satisfied as equalities in the sense of distributions, Bresch and Desjardins in \cite{bd} proposed some different assumptions from \cite{Feireisl}, and obtained the existence of global weak solutions to the full compressible Navier-Stokes equations with large initial data in $T^3$ or $\mathbb{R}^3$. Huang and Li in \cite{huangxiangdi4} established the global existence and uniqueness of classical solutions to the three-dimensional full compressible Navier-Stokes system in  $\mathbb{R}^3$ with smooth initial data
which are of small energy but possibly large oscillations where the initial density is allowed to vanish. Wen and Zhu in \cite{wenhuanyao} established the global existence of spherically and cylindrically symmetric classical and strong solutions of the full compressible Navier-Stokes equations in $\mathbb{R}^3$. The result in \cite{wenhuanyao} allowed the initial data may be large and the initial density may vanish. We also emphasis some blowup criterions in  \cite{huangxiangdi2}, \cite{huangxiangdi1}, \cite{huangxiangdi3}  (see also the reference therein). The reader may refer to \cite{chenqionglei} or \cite{wenhuanyao} for more recent advances on the subject.

Let us also recall that in the barotropic case, the critical Besov regularity was first considered by Danchin in an $L^2$ type framework to obtain a global solution \cite{D} for small perturbations of a stable constant state $\bar{\rho}$ with $\bar{\rho}>0$. Since then, there have been a number of refinements as regards admissible exponents for the global existence (see \cite{C.D}, \cite{C.M.Z1} and the references therein). The local-in-time existence issue in the critical regularity framework with both large $\rho_0$ and $u_0$ (with $\rho_0$ bounded away from 0) has been addressed only in the barotropic case. The proof either involves the time-weighted norm or the frequency localization techniques (see \cite{C.M.Z,D1} and \cite{H} for their generalization). The slightly homogeneous case (density close to some constant) is easier and has been investigated for the full Navier-Stokes equations as well in \cite{D2}. Lately, Chikami and Danchin in \cite{C-D} prove the local-in-time existence issue for the more general initial data $\rho_0-1\in \dot{B}^{\frac dp}_{p,1}(\mathbb{R}^d)$, $u_0\in \dot{B}^{-1+\frac dp}_{p,1}(\mathbb{R}^d)$ and $\theta_0\in  \dot{B}^{-2+\frac dp}_{p,1}(\mathbb{R}^d)$ with $\inf\limits_{x\in \mathbb{R}^d}\rho_0(x)>0$.

In the present paper, we will consider the Cauchy problem of (\ref{1.1}) in $\mathbb{R}^d, d=2,3,4$. We  assume that $c_v=1$, $P(\rho,\theta)=P_{e}(\rho)+R\rho\theta$, $q=-\kappa(\theta)\nabla \theta$, $P_{e}(\rho)$, $\kappa(\theta)$, $\mu(\rho)$ are smooth functions about $\rho$, $\theta$. Then the system \eqref{1.1} becomes
\begin{eqnarray}\label{1.1+1}
\left\{\begin{aligned}
&\partial_t\rho+\mathrm{div}(\rho u)=0, \\
&\partial_t(\rho u)+\mathrm{div}(\rho u\otimes u)+\nabla P=\mathrm{div} \mathbb{S}, \\
&\partial_t(\rho \theta)+\mathrm{div}(\rho u\theta)+\mathrm{div} q=\mathbb{S}:\nabla u-R\rho\theta\mathrm{div} u. \\
\end{aligned}\right.
\end{eqnarray}
Due to the term $\mathrm{div} q=\mathrm{div}(\kappa(\theta)\nabla \theta)=\kappa(\theta)\Delta\theta+\kappa'(\theta)|\nabla \theta|^2$, obviously, the system  (\ref{1.1+1}) has no scaling invariance compared with the compressible Navier-Stokes equations with constant viscosities.
Let $\bar{\rho}$
 be a fixed positive constant. We look for the solutions,
$(\rho(x, t), u(x, t), \theta(x, t))$ to the Cauchy problem for \eqref{1.1+1} with the far field behavior:
$$\rho(x, t)\to\bar{\rho},\quad\quad u(x, t)\to0,\quad\quad \theta(x, t)\to 0\quad \mathrm{as}
\quad|x|\to \infty, t>0,$$
and initial data:
$$\rho|_{t=0}=\rho_0(x),\quad u|_{t=0}=u_0(x),\quad \theta|_{t=0}=\theta_0(x),\quad \mathrm{in } \quad{\mathbb R}^d.$$
 For notational simplicity, we assume $\bar{\rho}=1$, $R=1$, $P'_{e}(1)=1$, $\mu(1)>0$ and $\kappa(0)=\kappa>0$. Substituting $\rho$ by $\rho+1$, we reduce the system $(\ref{1.1+1})$ to the following equations
\begin{eqnarray}\label{1.2}
\left\{\begin{aligned}
&\partial_t\rho+\mathrm{div} u=-\mathrm{div}(\rho u),\\
&\partial_t u+u\cdot \nabla u+\nabla \rho+\nabla \theta-\mu(1)(\Delta u+\nabla\mathrm{div} u)
=F_1(\rho,u,\theta),\\
&\partial_t\theta-\kappa\Delta \theta+u\cdot \nabla \theta =F_2(\rho,u,\theta),\\
&(\rho,u,\theta)|_{t=0}=(\rho_0,u_0,\theta_0),
\end{aligned}\right.
\end{eqnarray}
where
\begin{align*}
F_1(\rho,u,\theta)=&(\frac{\mu(\rho+1)}{\rho+1}-\mu(1))(\Delta u+\nabla\mathrm{div} u)
 +\frac{\theta}{\rho+1}\nabla \rho\nonumber\\
 &+\frac{\mu'(\rho+1)}{\rho+1}(\nabla u+(\nabla u)^T)\cdot\nabla \rho-(P'_e(\rho+1)-1)\nabla \rho,
\end{align*}
\begin{align*}
F_2(\rho,u,\theta)=&(\frac{\kappa(\theta)}{\rho+1}-\kappa)\Delta \theta+\frac{\mu(\rho+1)}{\rho+1}(\nabla u+(\nabla u)^T):\nabla u-\theta\mathrm{div} u \nonumber\\
&+\frac{\kappa'(\theta)}{\rho+1}|\nabla \theta|^2.
\end{align*}

When solving $\eqref{1.2}$, the main difficulty is that the system is only partially parabolic, owing to the mass conservation equation which is of hyperbolic type. This precludes any attempt to use the Banach fixed point theorem in a suitable space. As a matter of fact, global existence for small initial data may be proved through a suitable norm uniform bound estimate scheme and compactness methods. In this paper, we first use the Littlewood-Paley decomposition theory in Sobolev spaces to establish a losing energy estimate. Then we apply the  classical Friedrich's regularization method to build global approximate solutions and prove the existence of a solution by compactness arguments for the small initial data. Moreover, we can obtain the solution to the system (\ref{1.2}) is unique.

The main theorem of this paper reads as follows.
\begin{theo}\label{th1.1}
Let $d=2,3,4$, $s>\frac d2$. For any $\rho_0\in H^{s}(\mathbb{R}^d)$, $u_0\in H^{s}(\mathbb{R}^d)$, $\theta_0\in H^{s}(\mathbb{R}^d)$,
 there exists a constant $\eta>0$ such that
$$
||\rho_0||^2_{H^{s}(\mathbb{R}^d)}+||u_0||^2_{H^{s}(\mathbb{R}^d)}+||\theta_0||^2_{H^{s}(\mathbb{R}^d)}\leq \eta ,
$$
then the system $(\ref{1.2})$ has a unique global solution $(\rho,u,\theta)$ satisfying
\begin{align*}
&(\rho,u,\theta)\in \mathcal{C}(\mathbb{R}^+;H^{s}(\mathbb{R}^d)\times H^{s}(\mathbb{R}^d)\times H^{s}(\mathbb{R}^d)),
\\ &(\nabla \rho,\nabla u,\nabla \theta)\in L^2(\mathbb{R}^+;H^{s-1}(\mathbb{R}^d)\times H^{s}(\mathbb{R}^d)\times H^{s}(\mathbb{R}^d)).
\end{align*}
\end{theo}

The paper is organized as follows. In Section 2, we recall the Littlewood-Paley theory and give some properties of inhomogeneous Sobolev spaces. In Section 3, we deduce a prior estimates of solutions to the system (\ref{1.2}). In Section 4, we prove the global existence and uniqueness of the solution to the system (\ref{1.2}) by Fredrich's method  for the small initial data.

\vspace*{1em}

\noindent\textbf{Notations.} In the following, we denote by $(\cdot,\cdot)$ the $L^2$ scalar product. Given a Banach space $X$, we
denote its norm by $\|\cdot\|_{X}$. The symbol $A\lesssim B$ denotes that there exists a constant $c>0$ independent of $A$ and $B$, such that $A\leq c B$. The symbol $A\approx B$ represents $A\lesssim B$ and $B\lesssim A$.

\section{The Littlewood-Palely theory}

\quad \quad In this section, we are going to recall the dyadic partition of unity in the Fourier variable, the so-called, Littlewood-Paley theory, and the definition of Besov spaces. Part of the materials presented here can be found in \cite{B.C.D}. Let $\mathcal{S}(\mathbb{R}^d)$ be the Schwartz class of rapidly decreasing functions. For given $f\in \mathcal{S}(\mathbb{R}^d)$, its Fourier transform $\mathcal{F}f=\widehat{f}$ and its inverse Fourier transform $\mathcal{F}^{-1}f=\breve{f}$, respectively, defined by
\begin{align*}
\widehat{f}(\xi):=\int_{\mathbb{R}^d}e^{ix\cdot\xi}f(x)\mathrm{d} x,\quad \mathrm{and} \quad \breve{f}(x):=\frac{1}{(2\pi)^d}\int_{\mathbb{R}^d}e^{-ix\cdot\xi}f(\xi)\mathrm{d} \xi.
\end{align*}
Let $\varphi\in \mathcal{S}(\mathbb{R}^d)$ with values in $[0,1]$ such that $\varphi$ is supported in the ring $\mathcal{C}\triangleq \{\xi\in\mathbb{R}^d:\frac 3 4\leq|\xi|\leq \frac 8 3\}$ and $\chi$ is supported in the ring $\mathcal{B}\triangleq \{\xi\in\mathbb{R}^d:|\xi|\leq \frac 4 3\}$. Moreover,
\begin{align*}
\sum_{j\in\mathbb{Z}}\varphi(2^{-j}\xi)=1 \quad  \mathrm{for} \ \mathrm{any} \ \xi\in\mathbb{R}^d\setminus\{0\},
\end{align*}
and
\begin{align*}
\chi(\xi)+\sum_{j\geq0}\varphi(2^{-j}\xi)=1 \quad  \mathrm{for} \ \mathrm{any} \ \xi\in\mathbb{R}^d.
\end{align*}
Then for all $u \in \mathcal{S}'(\mathbb{R}^d)$, we can define the nonhomogeneous dyadic blocks as follows:
\begin{equation*}
\Delta_{-1}{u}\triangleq \chi(D)u=\mathcal{F}^{-1}(\chi \mathcal{F}u),\quad \Delta_j{u}\triangleq \varphi(2^{-j}D)u=\mathcal{F}^{-1}(\varphi(2^{-j}\cdot)\mathcal{F}u),\,\ \mathrm{if} \,\ j\geq 0.
\end{equation*}
We also can define the homogeneous dyadic blocks as follows:
\begin{equation*}
\dot{\Delta}_j{u}\triangleq \varphi(2^{-j}D)u=\mathcal{F}^{-1}(\varphi(2^{-j}\cdot)\mathcal{F}u),\,\ \mathrm{for} \,\ j\in \mathbb{Z}.
\end{equation*}
Hence, $ u={\sum\limits_{j\geq-1}}\Delta_j{u}$ in $\mathcal{S}'(\mathbb{R}^d)$ is called the nonhomogeneous Littlewood-Paley decomposition of $u$.

\begin{defi}
Let $s\in\mathbb{R}$, we set
\begin{equation*}
||f||_{H^s(\mathbb{R}^d)}\triangleq \Big(\sum_{j\geq-1}2^{2js}||\Delta_j{f}||^2_{L^2(\mathbb{R}^d)}\Big)^{\frac12}.
\end{equation*}
We define the homogeneous Hilbert space $H^s(\mathbb{R}^d)\triangleq \{f \in \mathcal{S}'(\mathbb{R}^d): ||f||_{H^s(\mathbb{R}^d)}<\infty\}$.
\end{defi}

\begin{rema}
Let $s\in\mathbb{R}$, we also set
\begin{equation*}
||f||_{\dot{H}^s(\mathbb{R}^d)}\triangleq \Big(\sum_{j\in \mathbb{Z}}2^{2js}||\dot{\Delta}_j{f}||^2_{L^2(\mathbb{R}^d)}\Big)^{\frac12}.
\end{equation*}
We can deduce that there exist two positive constants $c_0$ and $C_0$ such that
\begin{equation*}
c_0||f||_{\dot{H}^{s+1}(\mathbb{R}^d)}\leq ||\nabla f||_{\dot{H}^s(\mathbb{R}^d)}\leq C_0 ||f||_{\dot{H}^{s+1}(\mathbb{R}^d)}.
\end{equation*}
If $s>0$, then we have $||f||_{\dot{H}^s(\mathbb{R}^d)}\lesssim ||f||_{H^s(\mathbb{R}^d)}$.
\end{rema}
The following Bernstein's lemma will be repeatedly used throughout this paper.
\begin{lemm}\label{bernstein}
Let $\mathcal{B}$ be a ball and $\mathcal{C}$ a ring of $\mathbb{R}^d$. A constant $C$ exists so that for any positive real number $\lambda$, any
non-negative integer k, any smooth homogeneous function $\sigma$ of degree m, and any couple of real numbers $(a, b)$ with
$1\le a \le b$, there hold
\begin{align*}
&&\mathrm{Supp} \,\hat{u}\subset\lambda \mathcal{B}\Rightarrow\sup_{|\alpha|=k}\|\partial^{\alpha}u\|_{L^b}\le C^{k+1}\lambda^{k+d(\frac1a-\frac1b)}\|u\|_{L^a},\\
&&\mathrm{Supp} \,\hat{u}\subset\lambda \mathcal{C}\Rightarrow C^{-k-1}\lambda^k\|u\|_{L^a}\le\sup_{|\alpha|=k}\|\partial^{\alpha}u\|_{L^a}
\le C^{k+1}\lambda^{k}\|u\|_{L^a},\\
&&\mathrm{Supp} \,\hat{u}\subset\lambda \mathcal{C}\Rightarrow \|\sigma(D)u\|_{L^b}\le C_{\sigma,m}\lambda^{m+d(\frac1a-\frac1b)}\|u\|_{L^a}.
\end{align*}
\end{lemm}

To prove the main theorem, we need the following lemma concerning the product laws in Sobolev spaces, the proof of which is a standard application based on the  Littlewood-Paley theory.
\begin{lemm}\label{le2.3}(see \cite{B.C.D})
Let $\sigma>0$ and $\sigma_1\in \mathbb{R}$. Then we have for all $u,\ v\in H^{\sigma}(\mathbb{R}^d)\cap L^\infty(\mathbb{R}^d)$,
\begin{align*}
&||uv||_{H^\sigma(\mathbb{R}^d)}\lesssim ||u||_{H^\sigma(\mathbb{R}^d)}||v||_{L^\infty(\mathbb{R}^d)}+||v||_{H^\sigma(\mathbb{R}^d)}||u||_{L^\infty(\mathbb{R}^d)},\\
&||uv||_{\dot{H}^\sigma(\mathbb{R}^d)}\lesssim ||u||_{\dot{H}^\sigma(\mathbb{R}^d)}||v||_{L^\infty(\mathbb{R}^d)}+||v||_{\dot{H}^\sigma(\mathbb{R}^d)}||u||_{L^\infty(\mathbb{R}^d)}.
\end{align*}
Moreover, if $d\geq2$, then we have for $u\in H^\sigma(\mathbb{R}^d)\cap H^{\frac d2-1}(\mathbb{R}^d),\ v\in H^{\sigma+1}(\mathbb{R}^d)\cap L^\infty(\mathbb{R}^d)$,
\begin{align*}
||uv||_{\dot{H}^\sigma(\mathbb{R}^d)}\lesssim ||u||_{\dot{H}^\sigma(\mathbb{R}^d)}||v||_{L^\infty(\mathbb{R}^d)}+||u||_{\dot{H}^{\frac d2-1}(\mathbb{R}^d)}||v||_{\dot{H}^{\sigma+1}(\mathbb{R}^d)}.
\end{align*}
If $\sigma>\frac d2$, then $H^{\sigma}(\mathbb{R}^d)$ embeds into $L^\infty(\mathbb{R}^d)$. Also, for all $u,v\in H^\sigma(\mathbb{R}^d)$, there holds
\begin{align*}
||uv||_{H^\sigma(\mathbb{R}^d)}\lesssim ||u||_{H^\sigma(\mathbb{R}^d)}||v||_{H^\sigma(\mathbb{R}^d)}.
\end{align*}
Else if $\sigma_1\leq\frac d2<\sigma$ and $\sigma_1+\sigma>0$, then, for all $u\in H^\sigma(\mathbb{R}^d),\ v\in H^{\sigma_1}(\mathbb{R}^d)$, there holds
\begin{align*}
||uv||_{H^{\sigma_1}(\mathbb{R}^d)}\lesssim ||u||_{H^\sigma(\mathbb{R}^d)}||v||_{H^{\sigma_1}(\mathbb{R}^d)}.
\end{align*}
\end{lemm}

\begin{lemm}\label{le2.4}(see \cite{B.C.D})
Let $\sigma>0$ and $f$ be a smooth function such that $f(0)=0$. If $u\in H^\sigma(\mathbb{R}^d)$, then there exists a function $C=C(\sigma,f,d)$ such that
\begin{align*}
&||f(u)||_{H^\sigma(\mathbb{R}^d)}\leq C(||u||_{L^\infty(\mathbb{R}^d)})||u||_{H^\sigma(\mathbb{R}^d)}, \\
&||f(u)||_{\dot{H}^\sigma(\mathbb{R}^d)}\leq C(||u||_{L^\infty(\mathbb{R}^d)})||u||_{\dot{H}^\sigma(\mathbb{R}^d)}.
\end{align*}
\end{lemm}

\begin{lemm}\label{le2.5}(see \cite{B.C.D})
Let $\sigma>\frac d2$ and $f$ be a smooth function such that $f'(0)=0$. If $u,\ v\in H^\sigma(\mathbb{R}^d)$, then there exists a function $C=C(\sigma,f,d)$ such that
\begin{equation*}
||f(u)-f(v)||_{H^\sigma(\mathbb{R}^d)}\leq C(||u||_{L^\infty(\mathbb{R}^d)},||v||_{L^\infty(\mathbb{R}^d)})||u-v||_{H^\sigma(\mathbb{R}^d)}
(||u||_{H^\sigma(\mathbb{R}^d)}+||v||_{H^\sigma(\mathbb{R}^d)}).
\end{equation*}
\end{lemm}

\begin{lemm}\label{le2.6}(see \cite{B.C.D})
Let $\sigma>\frac d2-1$. There exists a positive sequence $\{c_{j}\}_{j\geq -1}$ satisfying $||c_{j}||_{\ell^2}=1$ such that
\begin{equation*}
||[u\cdot \nabla,\Delta_jf]||_{L^2(\mathbb{R}^d)}\lesssim c_{j}2^{-j(\sigma+1)}||\nabla u||_{H^{\sigma+1}(\mathbb{R}^d)}||f||_{H^{\sigma+1}(\mathbb{R}^d)} \quad \mathrm{for} \quad  j\geq -1.
\end{equation*}
\end{lemm}

\begin{lemm}\label{le2.7}
Let $\sigma>\frac d2-1$. Suppose that $\rho\in H^\sigma(\mathbb{R}^d)$ and $u\in H^{\sigma+1}(\mathbb{R}^d)$. Then we have for $j\geq 0$,
\begin{align*}
\Big|\int_{\mathbb{R}^d}\Delta_{j}\mathrm{div}(\rho u)\Delta(\Delta_{j}\rho)\mathrm{d} x\Big|\lesssim c^2_{j}2^{-2j\sigma}||\rho||_{H^{\sigma+1}(\mathbb{R}^d)}(||\nabla \rho||^2_{H^\sigma(\mathbb{R}^d)}+||\nabla u||^2_{H^{\sigma+1}(\mathbb{R}^d)}),
\end{align*}
where $\{c_{j}\}_{j\geq 0}$ satisfying $||c_{j}||_{\ell^2}\leq1$.
\end{lemm}

\begin{proof}
It is easy to get that
\begin{align*}
\int_{\mathbb{R}^d}\Delta_{j}\mathrm{div}(\rho u)\Delta(\Delta_{j}\rho)\mathrm{d} x=\int_{\mathbb{R}^d}\Delta_{j}(u\cdot\nabla \rho )\Delta(\Delta_{j}\rho)\mathrm{d} x+\int_{\mathbb{R}^d}\Delta_{j}(\rho \mathrm{div} u)\Delta(\Delta_{j}\rho)\mathrm{d} x.
\end{align*}
On the one hand, by Lemma \ref{le2.6}, we have
\begin{align*}
&\Big|\int_{\mathbb{R}^d}\Delta_{j}(u\cdot\nabla \rho )\Delta(\Delta_{j}\rho)\mathrm{d} x\Big|\\
\lesssim& \Big|\int_{\mathbb{R}^d}[u\cdot \nabla,\Delta_{j}]\rho )\Delta(\Delta_{j}\rho)\mathrm{d} x\Big|+\Big|\int_{\mathbb{R}^d}u\cdot\nabla \Delta_{j}\rho \Delta(\Delta_{j}\rho)\mathrm{d} x\Big|
\\ \lesssim& c^2_j2^{-2j\sigma}||\rho||_{H^{\sigma+1}(\mathbb{R}^d)}||\nabla \rho||_{H^\sigma(\mathbb{R}^d)}||\nabla u||_{H^{\sigma+1}(\mathbb{R}^d)}
\\& \quad \ +\Big|\int_{\mathbb{R}^d}\mathrm{div} u|\nabla \Delta_j\rho|^2\mathrm{d} x\Big|+\Big|\int_{\mathbb{R}^d}(\nabla u)^T:(\nabla \Delta_j\rho\otimes \nabla \Delta_j\rho)\mathrm{d} x\Big|
\\\lesssim& c^2_{j}2^{-2j\sigma}||\rho||_{H^{\sigma+1}(\mathbb{R}^d)}(||\nabla \rho||^2_{H^\sigma(\mathbb{R}^d)}+||\nabla u||^2_{H^{\sigma+1}(\mathbb{R}^d)}).
\end{align*}
On the other hand, it follows by Lemma \ref{le2.3} that
\begin{align*}
\Big|\int_{\mathbb{R}^d}\Delta_{j}(\rho \mathrm{div} u)\Delta(\Delta_j\rho)\mathrm{d} x\Big|&\lesssim c^2_j2^{-2j\sigma}||\nabla \rho||_{H^\sigma(\mathbb{R}^d)}||\rho||_{H^{\sigma+1}(\mathbb{R}^d)}||\nabla u||_{H^{\sigma+1}(\mathbb{R}^d)}
\\&\lesssim c^2_j2^{-2j\sigma}||\rho||_{H^{\sigma+1}(\mathbb{R}^d)}(||\nabla\rho||^2_{H^{\sigma}(\mathbb{R}^d)}+||\nabla u||^2_{H^{\sigma+1}(\mathbb{R}^d)}).
\end{align*}
Therefore, combining this two inequalities, then we complete the proof of the lemma.
\end{proof}

\section{A priori estimates}

\quad \quad In this section, we need to establish a losing energy estimate to the system (\ref{1.2}) which is motivated by \cite{W.X}. Firstly, we transform (\ref{1.2}) to the following linear system:
\begin{equation}\label{3.1}\begin{cases}
\partial_t\rho+\mathrm{div} u=G_1,\\
\partial_t u-\mu(1)\Delta u-\mu(1)\nabla\mathrm{div} u+\nabla \rho+\nabla \theta=G_2,\\
\partial_t\theta-\kappa\Delta \theta=G_3,\\
(\rho,u,\theta)|_{t=0}=(\rho_0,u_0,\theta_0),
\end{cases}\end{equation}
where
\begin{equation*}\begin{cases}
G_1=-\mathrm{div}(\rho u),\\
G_2=(\frac{\mu(\rho+1)}{\rho+1}-\mu(1))(\Delta u+\nabla \mathrm{div} u)\big)+\frac{\nabla \rho}{\rho+1}\theta
\\ \quad \quad \quad -u\cdot\nabla u+\frac{\mu'(\rho+1)}{\rho+1}(\nabla u+(\nabla u)^T)\cdot\nabla \rho-(P'_e(\rho+1)-1)\nabla \rho,\\
G_3=(\frac{\kappa(\theta)}{\rho+1}-\kappa)\Delta \theta+\frac{\mu(\rho+1)}{\rho+1}(\nabla u+(\nabla u)^T):\nabla u-\theta\mathrm{div} u-u\cdot \nabla \theta+\frac{\kappa'(\theta)}{\rho+1}|\nabla \theta|^2.
\end{cases}\end{equation*}
 In order to simplify the notation, we define the functional set $(\rho,u,\theta)\in E(T)$ if
$$(\rho,u,\theta)\in \mathcal{C}([0,T];H^{s}\times H^{s}\times H^{s}), \quad (\nabla \rho,\nabla u,\nabla \theta)\in L^2([0,T];H^{s-1} \times H^{s}\times H^{s}).$$
We also define norm as
\begin{equation*}
||(\rho,u,\theta)||_{E(0)}=||\rho_0||^2_{H^{s}}+||u_0||^2_{H^{s}}+||\theta_0||^2_{H^{s}},
\end{equation*}
\begin{align*}
||(\rho,u,\theta)||_{E(T)}&=||\rho||^2_{L^\infty_T(H^{s})}+||u||^2_{L^\infty_T(H^{s})}+||\theta||^2_{L^\infty_T(H^{s})}
\\& \quad \ +||\nabla \rho||^2_{L^2_T(H^{s-1})}+||\nabla u||^2_{L^2_T(H^{s})}+||\nabla \theta||^2_{L^2_T(H^{s})}.
\end{align*}
Then, we have the following main proposition.
\begin{prop}\label{th3.1}
Let $s>\frac d2$, $d=2,3,4$ and $T>0$. Let $(\rho,u,\theta)\in E(T)$ be the solution of the Cauchy problem $(\ref{3.1})$ with initial data $(\rho_0,u_0,\theta_0)$. Suppose that $||\rho(t,\cdot)||_{L^\infty}\leq \frac12$ and $||\theta(t,\cdot)||_{L^\infty}\leq \frac12$, then there exists a positive constant $\mathfrak{C}$ depending on $s$, $\mu(1)$, $\kappa$ and the smooth functions $\mu(x)$ and $\kappa(x)$ such that
\begin{align*}
||(\rho,u,\theta)||_{E(T)}\lesssim \mathfrak{C}||(\rho,u,\theta)||_{E(0)}+\mathfrak{C}||(\rho,u,\theta)||^2_{E(T)}+\mathfrak{C}||(\rho,u,\theta)||^3_{E(T)}.
\end{align*}
\end{prop}
\begin{proof}
We firstly apply the operator $\Delta_j$ to $(\ref{3.1})$  to get
\begin{equation}\label{3.2}\begin{cases}
\partial_t\rho_j+\mathrm{div} u_j=\Delta_jG_1,\\
\partial_t u_j-\mu(1)\Delta u_j-\mu(1)\nabla\mathrm{div} u_j+\nabla \rho_j+\nabla \theta_j=\Delta_jG_2,\\
\partial_t\theta_j-\kappa\Delta \theta_j=\Delta_jG_3,
\end{cases}\end{equation}
where $\rho_j=\Delta_j\rho$, $u_j=\Delta_ju$ and $\theta_j=\Delta_j\theta$. Let $\lambda<\frac12$ be a positive constant chosen be later. Multiplying the first equation of $(\ref{3.2})$ by $\rho_j-\Delta \rho_j-\lambda \mathrm{div} u_j$ and integrating by parts, it follows that
\begin{align}\label{ming1}
\frac12\frac{\mathrm{d}}{\mathrm{d} t}||\rho_j||^2_{H^1}-\lambda ||\mathrm{div} u_j||^2_{L^2}=\lambda(\partial_t\rho_j,\mathrm{div} u_j)-(\mathrm{div} u_j,\rho_j-\Delta \rho_j)+(\Delta_jG_1,\rho_j-\Delta \rho_j+\lambda \mathrm{div} u_j).
\end{align}
Multiplying the second equation of $(\ref{3.2})$ by $u_j-\Delta u_j+\lambda \nabla \rho_j$ and integrating by parts, we obtain
\begin{align}\label{ming2}
&\frac12\frac{\mathrm{d}}{\mathrm{d} t}||u_j||^2_{H^1}+\mu(1)||\nabla u_j||^2_{L^2}+\mu(1)||\Delta u_j||^2_{L^2}+\mu(1)||\mathrm{div} u_j||^2_{H^1}+\lambda||\nabla \rho_j||^2_{L^2}\nonumber\\
&\quad=-\lambda(\partial_tu_j,\nabla\rho_j)-(\nabla \rho_j,u_j-\Delta u_j)-(\nabla \theta_j,u_j-\Delta u_j)
+2\lambda\mu(1)(\Delta u_j,\nabla \rho_j)\nonumber\\
&\quad\quad-\lambda(\nabla\rho_j,\nabla \theta_j)+(\Delta_jG_2,u_j-\Delta u_j+\lambda \nabla \rho_j).
\end{align}
Multiplying the third equation of $(\ref{3.2})$ by $\theta_j-\Delta \theta_j$ and integrating by parts, we have
\begin{align}\label{ming3}
\frac12\frac{\mathrm{d}}{\mathrm{d} t}||\theta_j||^2_{H^1}+\kappa||\nabla \theta_j||^2_{L^2}+\kappa||\Delta \theta_j||^2_{L^2}=(\Delta_jG_3,\theta_j-\Delta \theta_j).
\end{align}
By using  the fact that
\begin{align*}
(\mathrm{div} u_j,\rho_j-\Delta \rho_j)+(\nabla \rho_j,u_j-\Delta u_j)=0,
\end{align*}
we can deduce from \eqref{ming1}--\eqref{ming3} that
\begin{align}\label{ming4}
& \ \frac12\frac{\mathrm{d}}{\mathrm{d} t}\Big(||\rho_j||^2_{H^1}+||u_j||^2_{H^1}+2\lambda(u_j,\nabla \rho_j)+\Lambda||\theta_j||^2_{H^1}\Big)+\mu(1)||\nabla u_j||^2_{L^2} +\mu(1)||\Delta u_j||^2_{L^2}
\nonumber\\
&\quad+\mu(1)||\mathrm{div} u_j||^2_{H^1} -\lambda||\mathrm{div} u_j||^2_{L^2}+\lambda||\nabla \rho_j||^2_{L^2}+\Lambda\kappa||\nabla \theta_j||^2_{L^2}+\Lambda\kappa||\Delta \theta_j||^2_{L^2}\nonumber\\
&=-(\nabla \theta_j,u_j-\Delta u_j)+2\lambda\mu(1)(\Delta u_j,\nabla \rho_j)-\lambda(\nabla\rho_j,\nabla\theta_j)\nonumber\\
&\quad+ (\Delta_jG_1,\rho_j-\Delta \rho_j+\lambda \mathrm{div} u_j)+(\Delta_jG_2,u_j-\Delta u_j+\lambda \nabla \rho_j)+\Lambda(\Delta_jG_3,\theta_j-\Delta \theta_j),
\end{align}
where $\Lambda$ is a large positive constant chosen be later. Note that for $j\geq 0$, we have $||u_j||_{L^2}\leq c_02^{-j}||\nabla u_j||_{L^2}$ for some $c_0>0$. Then, it is easy to get for $j\geq0$ that
\begin{align}\label{3.3}
||\Delta u_j||^2_{L^2}+||\nabla u_j||^2_{L^2}\approx ||\nabla u_j||^2_{H^1}, \quad ||\Delta \theta_j||^2_{L^2}+||\nabla \theta_j||^2_{L^2}\approx ||\nabla \theta_j||^2_{H^1},
\end{align}
\begin{align}\label{3.4}
||\rho_j||^2_{H^1}+||u_j||^2_{H^1}+2\lambda(u_j,\nabla \rho_j)+\Lambda||\theta_j||^2_{H^1}\approx ||\rho_j||^2_{H^1}+||u_j||^2_{H^1}+\Lambda||\theta_j||^2_{L^2},
\end{align}
\begin{align}\label{3.5}
&\quad \ |(\nabla \theta_j,u_j-\Delta u_j)|+|2\lambda\mu(1)(\Delta u_j,\nabla \rho_j)|+\lambda|(\nabla\rho_j,\nabla\theta_j)|\nonumber\\&\leq \frac12\mu(1)||\Delta u_j||^2_{L^2}+\frac12\mu(1)||\nabla u_j||^2_{L^2}+(4\lambda^2\mu(1)+\lambda^2)||\nabla\rho_j||^2_{L^2}+\frac{c^2_0+1+\mu(1)}{\mu(1)}||\nabla\theta_j||^2_{L^2}.
\end{align}

Choosing $\lambda$ small enough and $\Lambda$ big enough and combining $(\ref{3.3})-(\ref{3.5})$, we infer from \eqref{ming4} for $j\geq 0$ that
\begin{align}\label{3}
&\quad \ ||\rho_j||^2_{H^1}+||u_j||^2_{H^1}+||\theta_j||^2_{H^1}+\int^t_0(||\nabla \rho_j||^2_{L^2}+||\nabla u_j||^2_{H^1}+||\nabla \theta_j||^2_{H^1})\mathrm{d} \tau \nonumber
\\& \lesssim ||\Delta_j\rho_0||^2_{H^1}+||\Delta_j u_0||^2_{H^1}+||\Delta_j\theta_0||^2_{H^1}+\int^t_0\Big(|(\Delta_jG_1,\rho_j-\Delta \rho_j+\lambda \mathrm{div} u_j)| \nonumber     \\& \quad \ +|(\Delta_jG_2,u_j-\Delta u_j+\lambda \nabla \rho_j)|+\Lambda|(\Delta_jG_3,\theta_j-\Delta \theta_j)|\Big)\mathrm{d} \tau.
\end{align}
Thanks to the H$\ddot{\mathrm{o}}$lder inequality, we can also obtain for $j\geq 0$ that
\begin{align}\label{3.6}
\Lambda|(\Delta_jG_3,\theta_j-\Delta \theta_j)|\lesssim ||\Delta_jG_3||_{L^2}||\nabla \theta_j||_{H^1},
\end{align}
\begin{align}\label{3.7}
|(\Delta_jG_1,\rho_j+\lambda \mathrm{div} u_j)|\lesssim 2^{-j}||\Delta_jG_1||_{L^2}(||\nabla \rho_j||_{L^2}+||\nabla u_j||_{H^1}),
\end{align}
\begin{align}\label{3.8}
|(\Delta_jG_2,u_j-\Delta u_j+\lambda \nabla \rho_j)|\lesssim ||\Delta_jG_2||_{L^2}(||\nabla \rho_j||_{L^2}+||\nabla u_j||_{H^1}).
\end{align}

Using $(\ref{3.6})-(\ref{3.8})$ and applying Lemma \ref{le2.7}, the H\"{o}lder inequality, it follows that
\begin{align}\label{3.9}
&\quad \ ||\rho_j||^2_{H^1}+||u_j||^2_{H^1}+||\theta_j||^2_{H^1}+\int^t_0(||\nabla \rho_j||^2_{L^2}+||\nabla u_j||^2_{H^1}+||\nabla \theta_j||^2_{H^1})\mathrm{d} \tau \nonumber
\\& \lesssim ||\Delta_j\rho_0||^2_{H^1}+||\Delta_j u_0||^2_{H^1}+||\Delta_j\theta_0||^2_{H^1}+ \int^t_0c^2_j2^{-2js}||\rho||_{H^{s}}(||\nabla \rho||^2_{H^{s-1}}+||\nabla u||^2_{H^{s}})\mathrm{d} \tau \nonumber\\& \quad \
+\int^t_0(2^{-2j}||\Delta_jG_1||^2_{L^2}+||\Delta_jG_2||^2_{L^2}+||\Delta_jG_3||^2_{L^2})\mathrm{d} \tau.
\end{align}
By the same argument as in $(\ref{3})$, we can deduce that
\begin{align}\label{3.10}
& \quad \ ||\rho||^2_{H^1}+||u||^2_{H^1}+||\theta||^2_{H^1}+\int^t_0(||\nabla \rho||^2_{L^2}+||\nabla u||^2_{H^1}+||\nabla \theta||^2_{H^1})\mathrm{d} \tau \nonumber
\\& \lesssim ||\rho_0||^2_{H^1}+||u_0||^2_{H^1}+||\theta_0||^2_{H^1} +\int^t_0\Big(|(G_1,\rho-\Delta \rho+\lambda \mathrm{div} u)|+|(G_2,u-\Delta u+\lambda \nabla \rho)|
\nonumber \\& \quad \ +\Lambda|(G_3,\theta-\Delta \theta)|\Big)\mathrm{d} \tau.
\end{align}
Note that
\begin{align}\label{wan1}
&\quad \ ||\Delta_{-1}\rho||^2_{H^{s}}+||\Delta_{-1}u||^2_{H^{s}}+||\Delta_{-1}\theta||^2_{H^{s}} \nonumber
\\& \quad \ +\int^t_0(||\Delta_{-1}\nabla \rho||^2_{H^{s-1}}+||\Delta_{-1}\nabla u_||^2_{H^{s}}+||\Delta_{-1}\nabla \theta||^2_{H^{s}})\mathrm{d} \tau \nonumber
\\& \lesssim  ||\rho||^2_{H^1}+||u||^2_{H^1}+||\theta||^2_{H^1}+\int^t_0(||\nabla \rho||^2_{L^2}+||\nabla u||^2_{H^1}+||\nabla \theta||^2_{H^1})\mathrm{d} \tau.
\end{align}
Now, multiplying $(\ref{3.9})$ by $2^{2j(s-1)}$ and then summing for $j\geq0$, it follows by  (\ref{3.10}), (\ref{wan1}) that
\begin{align}\label{3.11}
& \quad \ ||\rho||^2_{H^{s}}+||u||^2_{H^{s}}+||\theta||^2_{H^{s}}+\int^t_0(||\nabla \rho||^2_{H^{s-1}}+||\nabla u||^2_{H^{s}}+||\nabla \theta||^2_{H^{s}})\mathrm{d} \tau \nonumber
\\& \lesssim ||\rho_0||^2_{H^{s}}+||u_0||^2_{H^{s}}+||\theta_0||^2_{H^s}+\int^t_0(||G_1||^2_{\dot{H}^{s-2}}+||G_2||^2_{\dot{H}^{s-1}}+||G_3||^2_{\dot{H}^{s-1}})\mathrm{d} \tau \nonumber\\& \quad \ +\int^t_0\Big(|(G_1,\rho-\Delta \rho+\lambda \mathrm{div} u)|+|(G_2,u-\Delta u+\lambda \nabla \rho)|+\Lambda|(G_3,\theta-\Delta \theta)|\Big)\mathrm{d} \tau \nonumber
\\& \quad \ + \int^t_0||\rho||_{H^{s}}(||\nabla \rho||^2_{H^{s-1}}+||\nabla u||^2_{H^{s}})\mathrm{d} \tau.
\end{align}
By $||\nabla u||_{L^\infty}\lesssim ||\nabla u||_{H^{s}}$ and Lemma \ref{le2.3}, one has
\begin{align}\label{299}
&\quad \ |(\mathrm{div}(\rho u),\Delta \rho)|=|(u\cdot\nabla \rho,\Delta \rho)|+|(\mathrm{div}(\rho\mathrm{div} u),\nabla \rho)| \nonumber\\&\lesssim \Big|\int_{\mathbb{R}^d}\mathrm{div} u |\nabla \rho|^2\mathrm{d} x\Big|+\Big|\int_{\mathbb{R}^d}\nabla u:(\nabla\rho\otimes \nabla \rho)\mathrm{d} x\Big|+||\nabla \rho||_{L^2}||\rho\mathrm{div} u||_{H^1}\nonumber
\\&\lesssim ||\nabla u||_{H^{s}}||\nabla \rho||^2_{L^2}+||\nabla \rho||_{L^2}||\rho||_{H^1}||\nabla u||_{H^{s}}\lesssim ||\nabla \rho||_{L^2}||\rho||_{H^1}||\nabla u||_{H^{s}}.
\end{align}
Note that
\begin{equation}\label{300}
\begin{cases}
||\rho u||_{L^2}\lesssim||\rho||^2_{L^4}+||u||^2_{L^4}\lesssim ||\nabla u||_{L^2}||u||_{L^2}+||\nabla \rho||_{L^2}||\rho||_{L^2}, \quad \mathrm{if} \quad d=2,\\
||\rho u||_{L^2}\lesssim||\rho||_{L^6}||u||_{L^3}\lesssim ||\nabla \rho||_{L^2}||u||_{H^1}, \quad \quad \quad \quad  \quad \quad  \quad \quad \quad \  \mathrm{if} \quad d=3,\\
||\rho u||_{L^2}\lesssim||\rho||_{L^4}||u||_{L^4}\lesssim ||\nabla \rho||_{L^2}||\nabla u||_{L^2}, \quad \quad \quad \quad  \quad \quad  \quad \quad \ \mathrm{if} \quad d=4.
\end{cases}\end{equation}
Therefore, according to (\ref{300}), we obtain
\begin{align}\label{301}
||\rho u||_{L^2}\lesssim (||\nabla u||_{L^2}+||\nabla \rho||_{L^2})(||u||_{H^1}+||\rho||_{H^1}).
\end{align}
Note that $||\rho(t,\cdot)||_{L^\infty}\leq \frac12$. Then, for any smooth function $f(x)$ such that $f(0)=0$, we also have
\begin{align}\label{302}
||f(\rho) u||_{L^2}\lesssim (||\nabla u||_{L^2}+||\nabla \rho||_{L^2})(||u||_{H^1}+||\rho||_{H^1}).
\end{align}
Also, by (\ref{301}), we get
\begin{equation*}
|(\mathrm{div}(\rho u),\rho+\lambda \mathrm{div} u)|\lesssim (||\nabla \rho||_{L^2}+||\nabla \mathrm{div} u||_{L^2})(||\nabla u||_{L^2}+||\nabla \rho||_{L^2})(||u||_{H^1}+||\rho||_{H^1}).
\end{equation*}
Combining (\ref{299}) and (\ref{301}), we have
\begin{align}\label{3.14}
|(G_1,\rho-\Delta \rho+\lambda \mathrm{div} u)|&\lesssim (||\nabla \rho||_{L^2}+||\nabla \mathrm{div} u||_{L^2})||\rho u||_{L^2}+|(G_1,\Delta \rho)|\nonumber
\\&\lesssim (||\nabla \rho||_{H^{s-1}}+||\nabla u||_{H^{s}})^2(||\rho||_{H^{s}}+||u||_{H^{s}}).
\end{align}
By the facts $||\rho(t,\cdot)||_{L^\infty}\leq \frac12$, $||\theta(t,\cdot)||_{L^\infty}\leq \frac12$ and Lemmas $\ref{le2.3}-\ref{le2.4}$, we get
\begin{align}\label{103}
& \quad \ \Big\|\frac{\kappa'(\theta)}{\rho+1}-\kappa'(0)\Big\|_{H^{s}}\nonumber\\&\lesssim \Big\|\frac{\kappa'(\theta)-\kappa'(0)}{\rho+1}\Big\|_{H^{s}}+\Big\|\kappa'(0)(\frac{1}{\rho+1}-1)\Big\|_{H^{s}} \nonumber
\\&\lesssim \Big\|(\kappa'(\theta)-\kappa'(0))(\frac{1}{\rho+1}-1)+\kappa'(\theta)-\kappa'(0)\Big\|_{H^{s}}+\Big\|\kappa'(0)(\frac{1}{\rho+1}-1)\Big\|_{H^{s}} \nonumber
\\&\lesssim ||\theta||_{H^{s}}+||\rho||_{H^{s}},
\end{align}
\begin{align}\label{104}
\Big\|\frac{\kappa(\theta)}{\rho+1}-\kappa\Big\|_{H^{s}}&\lesssim \Big\|\frac{\kappa(\theta)-\kappa}{\rho+1}\Big\|_{H^{s}}+||\kappa(\frac{1}{\rho+1}-1)\Big\|_{H^{s}} \nonumber
\\&\lesssim \Big\|(\kappa(\theta)-\kappa)(\frac{1}{\rho+1}-1)+\kappa(\theta)-\kappa\Big\|_{H^{s}}+\Big\|\kappa(\frac{1}{\rho+1}-1)\Big\|_{H^{s}} \nonumber
\\&\lesssim ||\theta||_{H^{s}}+||\rho||_{H^{s}},
\end{align}
\begin{equation}\label{106}\begin{cases}
\Big\|\frac{\nabla \rho}{\rho+1}\Big\|_{\dot{H}^{s-1}}\lesssim ||\nabla\ln(1+\rho)||_{\dot{H}^{s-1}}\lesssim ||\ln(1+\rho)||_{\dot{H}^{s}}\lesssim ||\rho||_{\dot{H}^{s}} \lesssim ||\nabla\rho||_{\dot{H}^{s-1}},\\
\Big\|\frac{\nabla \rho}{\rho+1}\Big\|_{\dot{H}^{\frac d2-1}}\lesssim ||\nabla\ln(1+\rho)||_{\dot{H}^{\frac d2-1}}\lesssim ||\ln(1+\rho)||_{\dot{H}^{\frac d2}}\lesssim ||\rho||_{\dot{H}^{\frac d2}} \lesssim ||\nabla\rho||_{\dot{H}^{\frac d2-1}}.
\end{cases}\end{equation}
Following the same argument as in (\ref{301}), we also have
\begin{align}\label{303}
|||u|^2||_{L^2}\lesssim ||u||_{H^1}||\nabla u||_{L^2}, \quad ||\theta u||_{L^2}\lesssim (||\nabla u||_{L^2}+||\nabla \theta||_{L^2})(||u||_{H^1}+||\theta||_{H^1}).
\end{align}
Applying Lemmas $\ref{le2.3}-\ref{le2.4}$ and using (\ref{302}), (\ref{303}), one has
\begin{align}\label{304}
|(G_2,-\Delta u+\lambda \nabla \rho)|&\lesssim (||\nabla \rho||_{L^2}+||\Delta u||_{L^2})\Big(\Big\|(\frac{\mu(\rho+1)}{\rho+1}-\mu(1))(\Delta u+\nabla \mathrm{div} u)\Big\|_{L^2} \nonumber
\\& \quad \ +||\frac{\nabla \rho}{\rho+1}\theta||_{L^2}+||u\cdot\nabla u||_{L^2}+||\frac{\mu'(\rho+1)}{\rho+1}(\nabla u+(\nabla u)^T)\cdot\nabla \rho||_{L^2} \nonumber
\\& \quad \ +||(P'_e(\rho+1)-1)\nabla \rho||_{L^2}\Big) \nonumber
\\&\lesssim (||\nabla \rho||_{L^2}+||\Delta u||_{L^2})(||\rho||_{H^{s}}||\nabla u||_{H^1}+||\nabla \rho||_{L^2}||\theta||_{H^{s}}\nonumber
\\& \quad \ +||u||_{H^{s}}||\nabla u||_{L^2}+||\nabla\rho||_{L^2}||\nabla u||_{H^{s}}+||\rho||_{H^{s}}||\nabla \rho||_{L^2}),
\end{align}
\begin{align}\label{305}
|(G_2,u)|&\lesssim \Big\|(\frac{\mu(\rho+1)}{\rho+1}-\mu(1))\cdot u\Big\|_{L^2}||\Delta u+\nabla \mathrm{div} u||_{L^2}
 +\Big\|\frac{\nabla \rho}{\rho+1}\Big\|_{L^2}||u\theta||_{L^2} \nonumber
 \\& \quad \ +|||u|^2||_{L^2}||\mathrm{div} u||_{L^2}+\Big\|\frac{\mu'(\rho+1)}{\rho+1}(\nabla u+(\nabla u)^T)\cdot\nabla \rho\Big\|_{L^2}||u||_{L^2} \nonumber
\\& \quad \ +||(P'_e(\rho+1)-1)\cdot u||_{L^2}||\nabla \rho||_{L^2} \nonumber
\\&\lesssim (||\nabla u||_{H^1}+||\nabla \rho||_{L^2})(||\nabla u||_{L^2}+||\nabla \rho||_{L^2})(||u||_{H^1}+||\rho||_{H^1}) \nonumber
\\& \quad \ +||\nabla \rho||_{L^2}(||\nabla u||_{L^2}+||\nabla \theta||_{L^2})(||u||_{H^1}+||\theta||_{H^1}) +||u||_{H^1}||\nabla u||^2_{L^2}\nonumber
\\& \quad \ +||\nabla\rho||_{L^2}||\nabla u||_{H^{s}}||u||_{L^2}.
\end{align}
Combining $(\ref{304})$ and $(\ref{305})$, we have
\begin{align}\label{3.15}
& \quad \ |(G_2,u-\Delta u+\lambda \nabla \rho)|\nonumber \\&\lesssim (||\nabla \rho||_{H^{s-1}}+||\nabla u||_{H^{s}}+||\nabla \theta||_{H^{s}})^2(||\rho||_{H^{s}}+||u||_{H^{s}}+||\theta||_{H^{s}}).
\end{align}
For any $1\leq q< \infty$, it is easy to see that
\begin{align*}
\Big\|\frac{\kappa(\theta)}{\rho+1}-\kappa\Big\|_{L^q}\lesssim \Big\|\frac{\kappa(\theta)-\kappa}{\rho+1}\Big\|_{L^q}+\Big\|\kappa(\frac{1}{\rho+1}-1)\Big\|_{L^q}\lesssim ||\theta||_{L^q}+||\rho||_{L^q},
\end{align*}
which along with the same argument as in (\ref{300}) leads to
\begin{align}\label{306}
\Big\|(\frac{\kappa(\theta)}{\rho+1}-\kappa)\theta\Big\|_{L^2}\lesssim (||\nabla \rho||_{L^2}+||\nabla \theta||_{L^2})(||\rho||_{H^1}+||\theta||_{H^1}).
\end{align}
According to Lemma \ref{le2.3} and using (\ref{103}), (\ref{306}), we get
\begin{align}\label{307}
|(G_3,\Delta \theta)|
&\lesssim ||\Delta \theta||_{L^2}\Big(\Big\|(\frac{\kappa(\theta)}{\rho+1}-\kappa)\Delta \theta\Big\|_{L^2}+\Big\|\frac{\mu(\rho+1)}{\rho+1}(\nabla u+(\nabla u)^T):\nabla u\Big\|_{L^2} \nonumber
\\& \quad \ +||\theta\mathrm{div} u||_{L^2}+||u\cdot \nabla \theta||_{L^2}+\Big\|\frac{\kappa'(\theta)}{\rho+1}|\nabla \theta|^2\Big\|_{L^2}\Big) \nonumber
\\&\lesssim ||\Delta \theta||_{L^2}||\Delta \theta||_{L^2}(||\rho||_{H^{s}}+||\theta||_{H^{s}})+||\Delta\theta||_{L^2}||\nabla u||_{L^2}||\nabla u||_{H^{s}} \nonumber
\\& \quad \ +||\Delta \theta||_{L^2}(||\theta||_{L^2}||\nabla u||_{H^{s}}+||u||_{H^{s}}||\nabla \theta||_{L^2})
+||\nabla\theta||_{H^{s}}||\nabla \theta||_{L^2}||\Delta \theta||_{L^2},
\end{align}
\begin{align}\label{308}
|(G_3,\theta)|&\lesssim \Big\|(\frac{\kappa(\theta)}{\rho+1}-\kappa)\theta\Big\|_{L^2}||\Delta \theta||_{L^2}+\Big\|\frac{\mu(\rho+1)}{\rho+1}\theta\Big\|_{L^2}||(\nabla u+(\nabla u)^T):\nabla u||_{L^2} \nonumber
\\& \quad \ +|||\theta|^2||_{L^2}||\mathrm{div} u||_{L^2}+||u\theta||_{L^2}||\nabla \theta||_{L^2}+\Big\|\frac{\kappa'(\theta)}{\rho+1}\theta||_{L^2}|||\nabla \theta|^2\Big\|_{L^2} \nonumber
\\&\lesssim ||\Delta \theta||_{L^2}(||\nabla \rho||_{L^2}+||\nabla \theta||_{L^2})(||\rho||_{H^1}+||\theta||_{H^1})+||\theta||_{L^2}||\nabla u||_{L^2}||\nabla u||_{H^{s}}\nonumber
\\& \quad \ +||\nabla \theta||_{L^2}||\theta||_{H^1}||\nabla u||_{L^2}+||\nabla \theta||_{L^2}(||\nabla u||_{L^2}+||\nabla \theta||_{L^2})(||u||_{H^1}+||\theta||_{H^1}) \nonumber
\\& \quad \ +||\theta||_{L^2}||\nabla \theta||_{H^{s}}||\nabla \theta||_{L^2}.
\end{align}
Summing up $(\ref{307}), (\ref{308})$, we obtain
\begin{align}\label{3.16}
& \quad \ |(G_3,\theta-\Delta \theta)|\nonumber\\&\lesssim (||\nabla \rho||_{H^{s-1}}+||\nabla u||_{H^{s}}+||\nabla \theta||_{H^{s}})^2(||\rho||_{H^{s}}+||u||_{H^{s}}+||\theta||_{H^{s}}).
\end{align}
Using Lemma \ref{le2.3} gives rise to
\begin{align}\label{3.17}
||G_1||^2_{\dot{H}^{s-2}}&\lesssim ||\rho u||^2_{\dot{H}^{s-1}}\lesssim ||u||^2_{\dot{H}^{s-1}}||\rho||^2_{L^\infty}+||\rho||^2_{\dot{H}^{s-1}}||u||^2_{L^\infty} \nonumber
\\&\lesssim (||\nabla \rho||^2_{H^{s-1}}+||\nabla u||^2_{H^{s}})(||\rho||^2_{H^{s}}+||u||^2_{H^{s}}).
\end{align}
By Lemmas $\ref{le2.3}-\ref{le2.4}$ and (\ref{106}), we deduce that
\begin{align} \label{3.18}
||G_2||^2_{\dot{H}^{s-1}}&\lesssim \Big\|(\frac{\mu(\rho+1)}{\rho+1}-\mu(1))(\Delta u+\nabla \mathrm{div} u)\Big\|^2_{\dot{H}^{s-1}}
+\Big\|\frac{\nabla \rho}{\rho+1}\theta\Big\|^2_{\dot{H}^{s-1}}+||u\cdot\nabla u||^2_{\dot{H}^{s-1}} \nonumber
\\& \quad \ +\Big\|\frac{\mu'(\rho+1)}{\rho+1}(\nabla u+(\nabla u)^T)\cdot\nabla \rho\Big\|^2_{\dot{H}^{s-1}}
 +||(P'_e(\rho+1)-1)\nabla \rho||^2_{\dot{H}^{s-1}} \nonumber
\\&\lesssim ||\rho||^2_{H^{s}}||\nabla u||^2_{H^{s}}+||\nabla\rho||^2_{\dot{H}^{s-1}}||\theta||^2_{L^\infty}+||\nabla \rho||^2_{H^{s-1}}||\theta||^2_{\dot{H}^{s}}+||\nabla u||^2_{H^{s-1}}||u||^2_{H^{s}}\nonumber
\\& \quad \ +||\nabla \rho||^2_{H^{s-1}}||\nabla u||^2_{H^{s}}(1+||\rho||^2_{H^{s}})+||\rho||^2_{H^{s}}||\nabla \rho||^2_{H^{s-1}}  \nonumber
\\&\lesssim (||\nabla \rho||^2_{H^{s-1}}+||\nabla u||^2_{H^{s}}+||\nabla \theta||^2_{H^{s}})(||\rho||^2_{H^{s}}+||u||^2_{H^{s}}+||\theta||^2_{H^{s}}) \nonumber
\\& \quad \ +||\nabla u||^2_{H^{s}}||\rho||^2_{H^{s}}||\nabla \rho||^2_{H^{s-1}}.
\end{align}
Applying Lemma \ref{le2.3} and combining $(\ref{103})-(\ref{104})$, one has
\begin{align} \label{3.19}
||G_3||^2_{\dot{H}^{s-1}}&\lesssim \Big\|(\frac{\kappa(\theta)}{\rho+1}-\kappa)\Delta \theta\Big\|^2_{\dot{H}^{s-1}}+\Big\|\frac{\mu(\rho+1)}{\rho+1}(\nabla u+(\nabla u)^T):\nabla u\Big\|^2_{\dot{H}^{s-1}}\nonumber
\\& \quad \ +||\theta\mathrm{div} u||^2_{\dot{H}^{s-1}} +||u\cdot \nabla \theta||^2_{\dot{H}^{s-1}}+\Big\|\frac{\kappa'(\theta)}{\rho+1}|\nabla \theta|^2\Big\|^2_{\dot{H}^{s-1}}\nonumber
\\&\lesssim (||\rho||^2_{H^{s}}+||\theta||^2_{H^{s}})||\nabla \theta||^2_{H^{s}}+(1+||\rho||^2_{H^{s}})||\nabla u||^2_{H^{s-1}}||\nabla u||^2_{H^{s}} \nonumber
\\&\quad \ +(1+||\rho||^2_{H^{s}}+||\theta||^2_{H^{s}})||\nabla \theta||^2_{H^{s-1}}||\nabla \theta||^2_{H^{s}}+||u||^2_{H^{s}}||\nabla \theta||^2_{H^{s-1}}\nonumber
\\& \quad \ +||\theta||^2_{H^{s}}||\nabla u||^2_{H^{s-1}} \nonumber
\\&\lesssim (||\nabla \rho||^2_{H^{s-1}}+||\nabla u||^2_{H^{s}}+||\nabla \theta||^2_{H^{s}})(||\rho||^2_{H^{s}}+||u||^2_{H^{s}}+||\theta||^2_{H^{s}})\nonumber
\\& \quad \ + (||\theta||^2_{H^{s}}||\rho||^2_{H^{s}}+||\theta||^2_{H^{s}}||\theta||^2_{H^{s}})||\nabla \theta||^2_{H^{s}}+||\rho||^2_{H^{s}}||u||^2_{H^{s}}||\nabla u||^2_{H^{s}}.
\end{align}
Therefore, substituting (\ref{3.14}), (\ref{3.15}), $(\ref{3.16})-(\ref{3.19})$ into $(\ref{3.11})$ and then using H\"{o}lder's inequality, we obtain for all $t\in[0,T]$
\begin{align*}
& \quad \ ||\rho(t)||^2_{H^{s}}+||u(t)||^2_{H^{s}}+||\theta(t)||^2_{H^{s}}+\int^t_0(||\nabla \rho||^2_{H^{s-1}}+||\nabla u||^2_{H^{s}}+||\nabla \theta||^2_{H^{s}})\mathrm{d} \tau \nonumber
\\& \lesssim ||\rho_0||^2_{H^{s}}+||u_0||^2_{H^{s}}+||\theta_0||^2_{H^{s}}
\\& \quad \quad +\int^t_0(||\rho||^2_{H^{s}}+||u||^2_{H^{s}}+||\theta||^2_{H^{s}})
 (||\nabla \rho||^2_{H^{s-1}}+||\nabla u||^2_{H^{s}}+||\nabla \theta||^2_{H^{s}})\mathrm{d} \tau
 \\& \quad \quad   +\int^t_0(||\rho||^2_{H^{s}}+||u||^2_{H^{s}}+||\theta||^2_{H^{s}})^2(||\nabla \rho||^2_{H^{s-1}}+||\nabla u||^2_{H^{s}}+||\nabla \theta||^2_{H^{s}})\mathrm{d} \tau,
\end{align*}
which implies
\begin{align*}
||(\rho,u,\theta)||_{E(T)}\lesssim ||(\rho,u,\theta)||_{E(0)}+||(\rho,u,\theta)||^2_{E(T)}+||(\rho,u,\theta)||^3_{E(T)}.
\end{align*}
This completes the proof of Proposition \ref{th3.1}.
\end{proof}

\section{Proof of Theorem 1.1}

\quad \quad This section is devoted to proving the global well-posedness of the system $(\ref{1.2})$ for the small initial data.

\textbf{Step 1: Construction of the approximate solutions.} We first construct the approximate solutions and the construction is based on the classical Friedrich's method as in \cite{B.C.D}. Define the smoothing operator
$$\mathcal{J}_{\varepsilon}f=\mathcal{F}^{-1}(1_{0\leq|\xi|\leq \frac{1}{\varepsilon}}\mathcal{F}f).$$
We consider the following approximate system of (\ref{1.2}) for $U_{\varepsilon}=(\rho_{\varepsilon},u_{\varepsilon},\theta_{\varepsilon})$:
\begin{equation}\label{4.1}
\frac{\partial_tU_{\varepsilon}}{\partial t}=\mathcal{G}(U_{\varepsilon}), \quad U_{\varepsilon}=\mathcal{J}_{\varepsilon}(\rho_0,u_0,\theta_0),
\end{equation}
where $\mathcal{G}_{\varepsilon}(U_{\varepsilon})=(\mathcal{G}^{(1)}_{\varepsilon}(U_{\varepsilon}),\mathcal{G}^{(2)}_{\varepsilon}(U_{\varepsilon}),\mathcal{G}^{(3)}_{\varepsilon}(U_{\varepsilon}))$ is defined by
\begin{eqnarray}
\left\{\begin{aligned}
\mathcal{G}^{(1)}_{\varepsilon}=&-\mathcal{J}_{\varepsilon}\mathrm{div}(\mathcal{J}_{\varepsilon}\rho_{\varepsilon} \mathcal{J}_{\varepsilon}u_{\varepsilon})-\mathrm{div}(\mathcal{J}_{\varepsilon}u_{\varepsilon}),\\
\mathcal{G}^{(2)}_{\varepsilon}=&-\mathcal{J}_{\varepsilon}(\mathcal{J}_{\varepsilon}u_{\varepsilon}\cdot \nabla \mathcal{J}_{\varepsilon}u_{\varepsilon})+\mu(1)(\Delta \mathcal{J}_{\varepsilon}u_{\varepsilon}+\nabla\mathrm{div} \mathcal{J}_{\varepsilon}u_{\varepsilon}) +\mathcal{J}_{\varepsilon}\Big(\frac{\nabla \mathcal{J}_{\varepsilon}\rho_{\varepsilon}}{\mathcal{J}_{\varepsilon}\rho_{\varepsilon}+1}\mathcal{J}_{\varepsilon}\theta_{\varepsilon}\Big)\\
& -\nabla \mathcal{J}_{\varepsilon}\rho_{\varepsilon}-\nabla \mathcal{J}_{\varepsilon}\theta_{\varepsilon} -\mathcal{J}_{\varepsilon}\big((P'_e(\mathcal{J}_{\varepsilon}\rho_{\varepsilon}+1)-1)\nabla \mathcal{J}_{\varepsilon}\rho_{\varepsilon}\big)+\mathcal{J}_{\varepsilon}\Big((\frac{\mu(\mathcal{J}_{\varepsilon}\rho_{\varepsilon}+1)}{\mathcal{J}_{\varepsilon}\rho_{\varepsilon}+1}-\mu(1))\\
&\quad\times(\Delta\mathcal{J}_{\varepsilon}u_{\varepsilon}+\nabla\mathrm{div} \mathcal{J}_{\varepsilon} u_{\varepsilon})+\frac{\mu'(\mathcal{J}_{\varepsilon}\rho_{\varepsilon}+1)}{\mathcal{J}_{\varepsilon}\rho_{\varepsilon}+1}(\nabla \mathcal{J}_{\varepsilon} u_{\varepsilon}+(\nabla \mathcal{J}_{\varepsilon} u_{\varepsilon})^T)\cdot\nabla \mathcal{J}_{\varepsilon}\rho_{\varepsilon}\Big),\\
\mathcal{G}^{(3)}_{\varepsilon}=&\mathcal{J}_{\varepsilon}((\frac{\kappa(\mathcal{J}_{\varepsilon}\theta_{\varepsilon})}{\mathcal{J}_{\varepsilon}\rho_{\varepsilon}+1}-\kappa)\Delta \mathcal{J}_{\varepsilon}\theta_{\varepsilon})+\mathcal{J}_{\varepsilon}\big(\frac{\mu(\mathcal{J}_{\varepsilon}\rho_{\varepsilon}+1)}{\mathcal{J}_{\varepsilon}\rho_{\varepsilon}+1}(\nabla \mathcal{J}_{\varepsilon}u_{\varepsilon}+(\nabla \mathcal{J}_{\varepsilon}u_{\varepsilon})^T):\nabla \mathcal{J}_{\varepsilon}u_{\varepsilon}\big)
\\
&  +\kappa\Delta \mathcal{J}_{\varepsilon}\theta_{\varepsilon}-\mathcal{J}_{\varepsilon}\big(\mathcal{J}_{\varepsilon}\theta_{\varepsilon}\mathrm{div} \mathcal{J}_{\varepsilon}u_{\varepsilon}\big)
-\mathcal{J}_{\varepsilon}\big(\mathcal{J}_{\varepsilon}u_{\varepsilon}\cdot \nabla \mathcal{J}_{\varepsilon}\theta_{\varepsilon}\big)+\mathcal{J}_{\varepsilon}\big(\frac{\kappa'(\mathcal{J}_{\varepsilon}\theta_{\varepsilon})}{\mathcal{J}_{\varepsilon}\rho_{\varepsilon}+1}|\nabla \mathcal{J}_{\varepsilon}\theta_{\varepsilon}|^2\big).
\end{aligned}\right.
\end{eqnarray}

Using the fact that $||\mathcal{J}_{\varepsilon}f||_{H^k}\lesssim (1+\frac{1}{\varepsilon^2})^{\frac k2}||f||_{L^2}$, it is easy to show that
$$
||\mathcal{G}_{\varepsilon}(U_{\varepsilon})||_{L^2}\leq C_{\varepsilon}f(||U_{\varepsilon}||_{L^2}),
$$
$$
||\mathcal{G}_{\varepsilon}(U_{\varepsilon})-\mathcal{G}_{\varepsilon}(\widetilde{U}_{\varepsilon})||_{L^2}\leq C_{\varepsilon}g(||U_{\varepsilon}||_{L^2},||\widetilde{U}_{\varepsilon}||_{L^2})\
||U_{\varepsilon}-\widetilde{U}_{\varepsilon}||_{L^2},
$$
where $f$ and $g$ are polynomials with positive coefficients. Therefore, the approximate system can be viewed as an ODE system on $L^2$. Then the Cauchy-Lipschitz theorem ensures that there exists a strictly maximal time $T_{\varepsilon}$ and a unique solution $(\rho_{\varepsilon},u_{\varepsilon},\theta_{\varepsilon})$, which is continuous in time with a value in $L^2$. As $\mathcal{J}^2_{\varepsilon}=\mathcal{J}_{\varepsilon}$, we know that $(\mathcal{J}_{\varepsilon}\rho_{\varepsilon},\mathcal{J}_{\varepsilon}u_{\varepsilon},\mathcal{J}_{\varepsilon}\theta_{\varepsilon})$ is also a solution of (\ref{4.1}). Therefore, $(\rho_{\varepsilon},u_{\varepsilon},\theta_{\varepsilon})=(\mathcal{J}_{\varepsilon}\rho_{\varepsilon},\mathcal{J}_{\varepsilon}u_{\varepsilon},\mathcal{J}_{\varepsilon}\theta_{\varepsilon})$. Thus, $(\rho_{\varepsilon},u_{\varepsilon},\theta_{\varepsilon})$ satisfies the following system
\begin{eqnarray}
\left\{\begin{aligned}
&\partial_t\rho_{\varepsilon}+\mathrm{div} u_{\varepsilon}=-\mathcal{J}_{\varepsilon}\mathrm{div}(\rho_{\varepsilon} u_{\varepsilon}),\nonumber\\
&\partial_t u_{\varepsilon}+\nabla \rho_{\varepsilon}+\nabla \theta_{\varepsilon}-\mu(1)(\Delta u_{\varepsilon}+\nabla\mathrm{div} u_{\varepsilon})=-\mathcal{J}_{\varepsilon}(u_{\varepsilon}\cdot \nabla u_{\varepsilon})-\mathcal{J}_{\varepsilon}\big((P'_e(\rho_{\varepsilon}+1)-1)\nabla \rho_{\varepsilon}\big)
\nonumber\\
 &\quad\quad+\mathcal{J}_{\varepsilon}\big(\frac{\nabla \rho_{\varepsilon}}{\rho_{\varepsilon}+1}\theta_{\varepsilon}\big) +\mathcal{J}_{\varepsilon}\big((\frac{\mu(\rho_{\varepsilon}+1)}{\rho_{\varepsilon}+1}-\mu(1))(\Delta u_{\varepsilon}+\nabla\mathrm{div} u_{\varepsilon})\big)
 \nonumber\\
 &\quad\quad+\mathcal{J}_{\varepsilon}\big(\frac{\mu'(\rho_{\varepsilon}+1)}{\rho_{\varepsilon}+1}(\nabla u_{\varepsilon}+(\nabla  u_{\varepsilon})^T)\cdot\nabla \rho_{\varepsilon}\big),\nonumber\\
&\partial_t\theta_{\varepsilon}-\kappa\Delta \theta_{\varepsilon}=\mathcal{J}_{\varepsilon}((\frac{\kappa(\theta_{\varepsilon})}{\rho_{\varepsilon}+1}-\kappa)\Delta \theta_{\varepsilon})+\mathcal{J}_{\varepsilon}\big(\frac{\mu(\rho_{\varepsilon}+1)}{\rho_{\varepsilon}+1}(\nabla u_{\varepsilon}+(\nabla u_{\varepsilon})^T):\nabla u_{\varepsilon}\big)
\nonumber\\
& \quad\quad -\mathcal{J}_{\varepsilon}(\theta_{\varepsilon}\mathrm{div} u_{\varepsilon})-\mathcal{J}_{\varepsilon}(u_{\varepsilon}\cdot \nabla \theta_{\varepsilon})
+\mathcal{J}_{\varepsilon}\big(\frac{\kappa'(\theta_{\varepsilon})}{\rho_{\varepsilon}+1}|\nabla \theta_{\varepsilon}|^2\big).
\end{aligned}\right.
\end{eqnarray}
Moreover, we also can conclude that $(\rho_{\varepsilon},u_{\varepsilon},\theta_{\varepsilon})\in E(T_{\varepsilon})$.

\textbf{Step 2: Uniform energy estimates.} Now, we will prove that $||(\rho_{\varepsilon},u_{\varepsilon},\theta_{\varepsilon})||_{E(T)}$ is uniformly bounded independent of $\varepsilon$ by the losing energy estimate. We assume that $\eta$ is small such that $||\rho_0||_{L^\infty}\leq \frac14$ and $||\theta_0||_{L^\infty}\leq \frac14$. Since the solution depends continuously on the time variable, then there exists a positive $T_0< T_{\varepsilon}$ such that the solution $(\rho_{\varepsilon},u_{\varepsilon},\theta_{\varepsilon})$ satisfies
$$||\rho_{\varepsilon}(t,\cdot)||_{L^\infty}\leq \frac12 \quad  \mathrm{and}  \quad ||\theta_{\varepsilon}(t,\cdot)||_{L^\infty}\leq \frac12 \quad  \mathrm{for} \quad  \mathrm{all} \quad  t\in[0,T_0],$$
$$||(\rho_{\varepsilon},u_{\varepsilon},\theta_{\varepsilon})||_{E(T_0)}\leq 2\mathfrak{C} ||(\rho,u,\theta)||_{E(0)}.$$
Without loss of generality, we assume that $T_0$ is a maximal time so that the above inequalities hold. In the following, we will give a refined estimate on $[0,T_0]$ for the solution. According to Proposition \ref{th3.1}, we obtain for all $0<T\leq T_0$,
\begin{align*}
||(\rho_{\varepsilon},u_{\varepsilon},\theta_{\varepsilon})||_{E(T)}&\leq \mathfrak{C} ||(\rho,u,\theta)||_{E(0)}+\mathfrak{C}||(\rho_{\varepsilon},u_{\varepsilon},\theta_{\varepsilon})||^2_{E(T)}+\mathfrak{C}||(\rho_{\varepsilon},u_{\varepsilon},\theta_{\varepsilon})||^3_{E(T)}
\\& \leq \mathfrak{C} ||(\rho,u,\theta)||_{E(0)}(1+4\mathfrak{C}^2\eta+8\mathfrak{C}^3\eta^2).
\end{align*}
Let $\eta<\frac{1}{4(\mathfrak{C}^2+1)}$. This implies
$$||(\rho_{\varepsilon},u_{\varepsilon},\theta_{\varepsilon})||_{E(T_0)}\leq \mathfrak{C} ||(\rho,u,\theta)||_{E(0)}(1+\frac12+\frac14)\leq \frac74\mathfrak{C} ||(\rho,u,\theta)||_{E(0)}< 2\mathfrak{C}\eta.$$
We also assume that $\eta$ is small enough such that $||\rho_{\varepsilon}(t,\cdot)||_{L^\infty}\leq \frac13$ and $||\theta_{\varepsilon}(t,\cdot)||_{L^\infty}\leq \frac13$. The standard bootstrap argument will ensure that, for all $0<T<\infty$, the following inequality holds
$$||(\rho_{\varepsilon},u_{\varepsilon},\theta_{\varepsilon})||_{E(T)}\leq 2\mathfrak{C} ||(\rho,u,\theta)||_{E(0)}.$$

\textbf{Step 3: Existence of the solution.} The existence of the solution $(\rho,u,\theta)\in E(T)$ of (\ref{1.2}) can be deduced by a standard compactness argument to the approximation sequence $(\rho_{\varepsilon},u_{\varepsilon},\theta_{\varepsilon})$. Moreover, there holds for all $0<T<\infty$,
$$||(\rho,u,\theta)||_{E(T)}\leq 2\mathfrak{C} ||(\rho,u,\theta)||_{E(0)}.$$

\textbf{Step 4: Uniqueness of the solution.} We will show that the solution guaranteed by Step 3 is unique. Suppose that $(\rho^1,u^1,\theta^1)\in E(T)$ and $(\rho^2,u^2,\theta^2)\in E(T)$ are two solutions of the system $(\ref{1.2})$. Denote $\delta \rho=\rho^1-\rho^2$, $\delta u=u^1-u^2$ and $\delta \theta=\theta^1-\theta^2$. Then, we have
\begin{equation}\label{4.2}\begin{cases}
\partial_t\delta\rho+\mathrm{div} \delta u=H_1,\\
\partial_t \delta u+\nabla \delta\rho+\nabla \delta\theta-\mu(1)(\Delta\delta u+\nabla\mathrm{div}\delta u)=H_2,\\
\partial_t\delta\theta-\kappa\Delta \delta\theta=H_3,
\end{cases}\end{equation}
where
$$
H_1=-\mathrm{div}(\delta\rho u^1)-\mathrm{div}(\rho^2\delta u),$$
\begin{align*}
H_2=&(\frac{\mu(\rho^1+1)}{\rho^1+1}-\mu(1))(\Delta \delta u+\nabla\mathrm{div}\delta u)+(\frac{\mu(\rho^1+1)}{\rho^1+1}-\frac{\mu(\rho^2+1)}{\rho^2+1})(\Delta u^2+\nabla\mathrm{div} u^2)
\\
& +\frac{\nabla \rho^1}{\rho^1+1}\delta\theta+(\frac{\nabla \rho^1}{\rho^1+1}-
\frac{\nabla \rho^2}{\rho^2+1})\theta^2+\frac{\mu'(\rho^1+1)}{\rho^1+1}(\nabla u^1+(\nabla u^1)^T)\cdot\nabla \delta\rho
\\& +\frac{\mu'(\rho^1+1)}{\rho^1+1}(\nabla \delta u+(\nabla\delta u)^T)\cdot\nabla \rho^2+(\frac{\mu'(\rho^1+1)}{\rho^1+1}-\frac{\mu'(\rho^2+1)}{\rho^2+1})(\nabla u^2+(\nabla u^2)^T)\cdot\nabla \rho^2\\
& -u^1\cdot \nabla\delta u-\delta u\cdot \nabla u^2-(P'_e(\rho^1+1)-1)\nabla \delta\rho-(P'_e(\rho^1+1)-P'_e(\rho^2+1))\nabla \rho^2,
\end{align*}
\begin{align*}
H_3=&(\frac{\kappa(\theta^1)}{\rho^1+1}-\kappa)\Delta \delta\theta+(\frac{\kappa(\theta^1)}{\rho^1+1}-\frac{\kappa(\theta^2)}{\rho^2+1})\Delta \theta^2-\theta^1\mathrm{div} \delta u-\delta \theta \mathrm{div} u^2-u^1\cdot \nabla \delta \theta-\delta u\cdot \nabla \theta^2
\\ & +\frac{\mu(\rho^1+1)}{\rho^1+1}(\nabla u^1+(\nabla u^1)^T):\nabla \delta u+\frac{\mu(\rho^1+1)}{\rho^1+1}(\nabla \delta u+(\nabla \delta u)^T):\nabla u^2
\\ & +(\frac{\mu(\rho^1+1)}{\rho^1+1}-\frac{\mu(\rho^2+1)}{\rho^2+1})(\nabla u^2+(\nabla  u^2)^T):\nabla u^2+\frac{\kappa'(\theta^1)}{\rho^1+1}\nabla (\theta^1+\theta^2)\cdot\nabla \delta\theta
\\ &+(\frac{\kappa'(\theta^1)}{\rho^1+1}-\frac{\kappa'(\theta^2)}{\rho^2+1})|\nabla \theta^2|^2.
\end{align*}
Multiplying the first equation of $(\ref{4.2})$ by $\delta\rho$ and integrating by parts, it follows that
\begin{align}\label{4.4}
\frac12\frac{\mathrm{d}}{\mathrm{d} t}||\delta\rho||^2_{L^2}+(\mathrm{div}\delta u,\delta\rho)=(H_1,\delta\rho).
\end{align}
Multiplying the second equation of $(\ref{4.2})$ by $\delta u$ and integrating by parts, we have
\begin{align}\label{4.5}
\frac12\frac{\mathrm{d}}{\mathrm{d} t}||\delta u||^2_{L^2}+(\nabla \delta\rho,\delta u)+(\nabla \delta \theta,\delta u)+\mu(1)(||\nabla \delta u||^2_{L^2}+||\mathrm{div} \delta u||^2_{L^2})=(H_2,\delta u).
\end{align}
Multiplying the third equation of $(\ref{4.2})$ by $\delta \theta$ and integrating by parts, we get
\begin{align}\label{4.6}
\frac12\frac{\mathrm{d}}{\mathrm{d} t}||\delta \theta||^2_{L^2}+\kappa||\nabla\delta\theta||^2_{L^2}=(H_3,\delta \theta).
\end{align}
In view of $(\mathrm{div}\delta u,\delta\rho)+(\nabla \delta\rho,\delta u)=0,$
we can deduce from $(\ref{4.4})-(\ref{4.6})$ that
\begin{align}\label{4.7}
& \quad \ \frac12\frac{\mathrm{d}}{\mathrm{d} t}(||\delta \rho||^2_{L^2}+||\delta u||^2_{L^2}+||\delta \theta||^2_{L^2})+\kappa||\nabla \delta \theta||^2_{L^2}+\mu(1)(||\nabla \delta u||^2_{L^2}+||\mathrm{div} \delta u||^2_{L^2})\nonumber\\
& =-(\nabla \delta \theta,\delta u)+(H_1,\delta \rho)+(H_2,\delta u)+(H_3,\delta \theta).
\end{align}
According to the Sobolev embedding relation, for all $a\in H^1,b\in H^{s-1}$, we have
\begin{align}\label{4.3}
||ab||_{L^2}\lesssim ||b||_{H^{s-1}}||a||_{H^1} \quad \mathrm{for} \quad d=2,3,4.
\end{align}
Thanks to Lemmas $\ref{le2.3}-\ref{le2.5}$, H$\ddot{\mathrm{o}}$lder's inequality and (\ref{4.3}), one has
\begin{align}\label{4.9}
|(H_1,\delta \rho)|&\lesssim \Big|\int_{\mathbb{R}^d}\mathrm{div} u^1|\delta \rho|^2\mathrm{d} x\Big|+\Big|\int_{\mathbb{R}^d}\rho^2\mathrm{div} \delta u\delta\rho\mathrm{d} x\Big|+\Big|\int_{\mathbb{R}^d}\delta u\nabla\rho^2\delta\rho\mathrm{d} x\Big| \nonumber
\\& \lesssim ||\nabla u^1||_{H^{s}}||\delta \rho||^2_{L^2}+||\rho^2||_{H^{s}}||\nabla \delta u||_{L^2}||\delta \rho||_{L^2}+||\delta \rho||_{L^2}||\nabla \rho^2||_{H^{s-1}}||\delta u||_{H^1} \nonumber
\\& \lesssim ||\nabla u^1||_{H^{s}}||\delta \rho||^2_{L^2}+||\rho^2||_{H^{s}}||\delta u||_{H^1}||\delta \rho||_{L^2}.
\end{align}
On the one hand, it follows by Lemmas $\ref{le2.3}-\ref{le2.5}$ and (\ref{4.3}) that
\begin{align}\label{xiao1}
|(u^1\cdot \nabla \delta u+\delta u\cdot \nabla u^2,\delta u)|\lesssim (||\nabla u^1||_{H^{s}}+||\nabla u^2||_{H^{s}})||\delta u||^2_{L^2},
\end{align}
\begin{align}\label{xiao2}
& \quad \ \Big|\Big(\frac{\mu'(\rho^1+1)}{\rho^1+1}(\nabla \delta u+(\nabla\delta u)^T)\cdot\nabla \rho^2,\delta u\Big)\Big|\lesssim ||\delta u||_{H^1}||\nabla \delta u||_{L^2}||\nabla \rho^2||_{H^{s-1}}\nonumber\\&\lesssim ||\rho^2||_{H^{s}}||\delta u||_{H^1}||\nabla \delta u||_{L^2}\lesssim \eta ||\nabla \delta u||^2_{L^2}+||\rho^2||_{H^{s}}||\delta u||_{L^2}||\nabla \delta u||_{L^2},
\end{align}
\begin{align}\label{xiao5}
\Big|(\frac{\nabla \rho^1}{\rho^1+1}\delta\theta,\delta u)\Big|\lesssim ||\nabla \rho^1||_{H^{s-1}}||\delta u||_{H^1}||\delta \theta||_{L^2}\lesssim || \rho^1||_{H^{s}}||\delta u||_{H^1}||\delta \theta||_{L^2},
\end{align}
\begin{align}\label{xiao3}
& \quad \ \left|\left((\frac{\mu(\rho^1+1)}{\rho^1+1}-\mu(1))(\Delta \delta u+\nabla\mathrm{div}\delta u),\delta u\right)\right|\nonumber\\&
\lesssim \left(\frac{\mu(\rho^1+1)}{\rho^1+1}-\mu(1),|\nabla \delta u|^2+|\mathrm{div} \delta u|^2\right)+\left(\nabla(\frac{\mu(\rho^1+1)}{\rho^1+1})\otimes  \delta u,\nabla \delta u\right)
\nonumber\\& \quad \ +\left(\delta u\cdot \nabla(\frac{\mu(\rho^1+1)}{\rho^1+1}),\mathrm{div}\delta u\right)
\nonumber\\&\lesssim \Big\|\frac{\mu(\rho^1+1)}{\rho^1+1}-\mu(1)\Big\|_{H^{s}}||\nabla \delta u||^2_{L^2}+||\delta u||_{H^1}||\nabla \delta u||_{L^2}\Big\|\nabla(\frac{\mu(\rho^1+1)}{\rho^1+1}-\mu(1))\Big\|_{H^{s-1}}
\nonumber\\& \lesssim ||\rho^1||_{H^{s}}||\delta u||_{H^1}||\nabla \delta u||_{L^2}\lesssim \eta ||\nabla \delta u||^2_{L^2}+||\rho^1||_{H^{s}}||\delta u||_{L^2}||\nabla \delta u||_{L^2},
\end{align}
\begin{align}\label{xiao4}
&\quad \ \Big|\Big((\frac{\mu(\rho^1+1)}{\rho^1+1}-\frac{\mu(\rho^2+1)}{\rho^2+1})(\Delta u^2+\nabla\mathrm{div} u^2),\delta u\Big)\Big|
\nonumber\\&\lesssim \Big\|\frac{\mu(\rho^1+1)}{\rho^1+1}-\frac{\mu(\rho^2+1)}{\rho^2+1}\Big\|_{L^2}||\delta u||_{H^1}||\Delta u^2||_{H^{s-1}}
\lesssim ||\nabla u^2||_{H^{s}}||\delta \rho||_{L^2}||\delta u||_{H^1},
\end{align}

\begin{align}\label{xiao6}
\quad \ \Big|\Big((\frac{\nabla \rho^1}{\rho^1+1}-\frac{\nabla \rho^2}{\rho^2+1})\theta^2,\delta u\Big)\Big|=&|(\ln(\rho^1+1)-\ln(\rho^2+1),\mathrm{div}(\theta^2\delta u))|
\nonumber\\
&\lesssim ||\delta \rho||_{L^2}||\theta^2\delta u||_{H^1} \lesssim ||\theta^2||_{H^{s}}||\delta u||_{H^1}||\delta \rho||_{L^2},
\end{align}
\begin{align}\label{xiao7}
\Big|\Big(\frac{\mu'(\rho^1+1)}{\rho^1+1}(\nabla u^1+(\nabla u^1)^T)\cdot\nabla \delta\rho,\delta u\Big)\Big|&\lesssim ||\nabla \delta \rho||_{H^{-1}}(1+||\rho^1||_{H^{s}})||\nabla u^1||_{H^{s}}||\delta u||_{H^1}\nonumber\\& \lesssim (1+||\rho^1||_{H^{s}})||\nabla u^1||_{H^{s}}||\delta u||_{H^1}||\delta \rho||_{L^2},
\end{align}

\begin{align}\label{xiao8}
\Big|\Big((\frac{\mu'(\rho^1+1)}{\rho^1+1}-\frac{\mu'(\rho^2+1)}{\rho^2+1})(\nabla u^2+(\nabla u^2)^T)\cdot\nabla \rho^2,\delta u\Big)\Big|\lesssim ||\delta \rho||_{L^2}||\delta u||_{H^1}||\nabla \rho^2||_{H^{s-1}}||\nabla u^2||_{H^{s}},
\end{align}
\begin{align}\label{xiao9}
& \quad \ |((P'_e(\rho^1+1)-1)\nabla \delta\rho+(P'_e(\rho^1+1)-P'_e(\rho^2+1))\nabla \rho^2,\delta u)|\nonumber\\&\lesssim ||\rho^1||_{H^{s}}||\nabla \delta \rho||_{H^{-1}}||\delta u||_{H^1}+||\delta \rho||_{L^2}||\nabla \rho^2||_{H^{s-1}}||\delta u||_{H^1}\nonumber\\&\lesssim (||\rho^1||_{H^{s}}+||\rho^2||_{H^{s}})||\delta\rho||_{L^2}||\delta u||_{H^1}.
\end{align}
Combining the above inequalities \eqref{xiao1}--\eqref{xiao9}, we have
\begin{align}\label{4.10}
|(H_2,\delta u)|\lesssim \eta||\nabla \delta u||^2_{L^2}+\mathcal{P}(t)||\delta u||_{H^1}(||\delta \rho||_{L^2}+||\delta u||_{L^2}+||\delta \theta||_{L^2}),
\end{align}
where $\mathcal{P}(t)$ depends on $||\rho^1||_{H^{s}}$, $||\rho^2||_{H^{s}}$, $||u^1||_{H^{s+1}}$, $||u^2||_{H^{s+1}}$, $||\theta^1||_{H^{s+1}}$, $||\theta^2||_{H^{s+1}}$.

On the other hand, by Lemmas $\ref{le2.3}-\ref{le2.5}$ and (\ref{104}), (\ref{4.3}), we also obtain
\begin{align}\label{xiao10}
& \quad \ |(\theta^1\mathrm{div} \delta u+\delta \theta \mathrm{div} u^2+u^1\cdot \nabla \delta \theta+\delta u\cdot \nabla \theta^2,\delta\theta)|
\nonumber\\&\lesssim (||u^1||_{H^{s}}+||u^2||_{H^{s}}+||\theta^1||_{H^{s}}+||\theta^2||_{H^{s}})(||\delta u||_{H^1}+||\delta \theta||_{H^1})||\delta \theta||_{L^2},
\end{align}
\begin{align}\label{xiao11}
\Big|\Big((\frac{\kappa(\theta^1)}{\rho^1+1}-\frac{\kappa(\theta^2)}{\rho^2+1})\Delta \theta^2,\delta \theta\Big)\Big|&\lesssim (||\delta\theta||_{L^2}+||\delta\rho||_{L^2})||\delta\theta||_{H^1}||\Delta\theta^2||_{H^{s-1}}
\nonumber\\&\lesssim ||\nabla \theta^2||_{H^{s}}(||\delta\rho||_{L^2}+||\delta \theta||_{L^2})||\delta\theta||_{H^1},
\end{align}
\begin{align}\label{xiao12}
 \Big|\Big((\frac{\kappa'(\theta^1)}{\rho^1+1}-\frac{\kappa'(\theta^2)}{\rho^2+1})|\nabla \theta^2|^2,\delta\theta\Big)\Big|
\lesssim ||\nabla\theta^2||^2_{H^{s}}(||\delta \rho||_{L^2}+||\delta \theta||_{L^2})||\delta\theta||_{L^2},
\end{align}
\begin{align}\label{xiao13}
&\quad \ \Big|\Big(\frac{\mu(\rho^1+1)}{\rho^1+1}(\nabla u^1+(\nabla u^1)^T):\nabla \delta u+\frac{\mu(\rho^1+1)}{\rho^1+1}(\nabla \delta u+(\nabla \delta u)^T):\nabla u^2,\delta \theta\Big)\Big|\nonumber\\&
\lesssim (||\nabla u^1||_{H^{s}}+||\nabla u^2||_{H^{s}})||\nabla \delta u||_{L^2}||\delta \theta||_{L^2},
\end{align}
\begin{align}\label{xiao14}
\Big|\Big(\frac{\kappa'(\theta^1)}{\rho^1+1}\nabla(\theta^1+\theta^2)\cdot\nabla\delta\theta,\delta\theta\Big)\Big|\lesssim (||\nabla \theta^1||_{H^{s}}+||\nabla \theta^2||_{H^{s}})||\delta\theta||_{L^2}||\nabla\delta\theta||_{L^2},
\end{align}
\begin{align}\label{xiao15}
\Big|\Big((\frac{\mu(\rho^1+1)}{\rho^1+1}-\frac{\mu(\rho^2+1)}{\rho^2+1})(\nabla u^2+(\nabla  u^2)^T):\nabla u^2,\delta \theta\Big)\Big|\lesssim ||\nabla u^2||^2_{H^{s}}||\delta \rho||_{L^2}||\delta\theta||_{H^1},
\end{align}
\begin{align}\label{xiao16}
\Big|\Big((\frac{\kappa(\theta^1)}{\rho^1+1}-\kappa)\Delta \delta\theta,\delta\theta\Big)\Big|&\lesssim |(\delta\theta\nabla(\frac{\kappa(\theta^1)}{\rho^1+1}-\kappa),\nabla \delta\theta)|+|(\frac{\kappa(\theta^1)}{\rho^1+1}-\kappa,|\nabla \delta\theta|^2)| \nonumber
\nonumber\\& \lesssim \Big\|\nabla(\frac{\kappa(\theta^1)}{\rho^1+1}-\kappa)\Big\|_{H^{s-1}}||\nabla \delta \theta||_{L^2}||\delta \theta||_{H^1}
+\Big\|\frac{\kappa(\theta^1)}{\rho^1+1}-\kappa\Big\|_{H^{s}}||\nabla \delta \theta||^2_{L^2} \nonumber
\nonumber\\& \lesssim (||\theta^1||_{H^{s}}+||\rho^1||_{H^{s}})||\delta \theta||_{H^1}||\nabla \delta\theta||^2_{L^2}
\nonumber\\& \lesssim \eta ||\nabla \delta\theta||^2_{L^2}+(||\theta^1||_{H^{s}}+||\rho^1||_{H^{s}})||\delta \theta||_{L^2}||\nabla \delta\theta||_{L^2}.
\end{align}

Combining the above inequalities \eqref{xiao10}-\eqref{xiao16}, we have
\begin{align}\label{4.11}
 |(H_3,\delta\theta)|\lesssim \eta||\nabla\delta \theta||^2_{L^2}+\mathcal{Q}(t)(||\delta u||_{H^1}+||\delta \theta||_{H^1})(||\delta \rho||_{L^2}+||\delta u||_{L^2}+||\delta \theta||_{L^2}),
\end{align}
where $\mathcal{Q}(t)$ depends on $||\rho^1||_{H^{s}}$, $||\rho^2||_{H^{s}}$, $||u^1||_{H^{s+1}}$, $||u^2||_{H^{s+1}}$, $||\theta^1||_{H^{s+1}}$, $||\theta^2||_{H^{s+1}}$.

Therefore, inserting (\ref{4.9}), (\ref{4.10}) and (\ref{4.11}) into (\ref{4.7}) and choosing $\eta$ small enough, we finally deduce that
\begin{align}\label{4.12}
&\quad \ \frac{\mathrm{d}}{\mathrm{d} t}(||\delta \rho||^2_{L^2}+||\delta u||^2_{L^2}+||\delta \theta||^2_{L^2})+\kappa||\nabla \delta \theta||^2_{L^2}+\mu(1)(||\nabla \delta u||^2_{L^2}+||\mathrm{div} \delta u||^2_{L^2})\nonumber
\\& \lesssim \mathcal{R}(t)(||\delta u||_{H^1}+||\delta \rho||_{H^1})(||\delta \rho||_{L^2}+||\delta u||_{L^2}+||\delta \theta||_{L^2}) \nonumber
\\&\lesssim (\mathcal{R}(t)+\frac{1}{\varepsilon}\mathcal{R}^2(t))(||\delta \rho||^2_{L^2}+||\delta u||^2_{L^2}+||\delta \theta||^2_{L^2})+\varepsilon(||\nabla \delta u||^2_{L^2}+||\nabla \delta \theta||^2_{L^2}),
\end{align}
where $\mathcal{R}(t)$ depends on $||\rho^1||_{H^{s}}$, $||\rho^2||_{H^{s}}$, $||u^1||_{H^{s+1}}$, $||u^2||_{H^{s+1}}$, $||\theta^1||_{H^{s+1}}$, $||\theta^2||_{H^{s+1}}$. Now let us choose $\varepsilon$ small enough and according to the Gronwall inequality, then (\ref{4.12}) implies $(\delta\rho(t),\delta u(t),\delta\theta(t))=(0,0,0)$ for all $t\in[0,\infty)$. This completes the proof of the uniqueness about the solution to the system (\ref{1.2}).

\vspace*{1em}
\noindent\textbf{Acknowledgements.} This work was
partially supported by NNSFC (No. 11271382), RFDP (No. 20120171110014), MSTDF (No. 098/2013/A3), Guangdong Special Support Program (No. 8-2015) and the key project of NSF of Guangdong Province (No. 1614050000014).
%\vspace*{1em}

\end{document}